\documentclass[11pt]{article}

\usepackage{enumerate}
\usepackage{amsmath}
\usepackage{amssymb,latexsym}
\usepackage{amsthm}
\usepackage{color}
\usepackage{graphicx}
\usepackage{hyperref}
\usepackage{amscd}
\usepackage{lineno,xcolor}
\DeclareMathOperator{\Tr}{Tr}
\DeclareMathOperator{\N}{N}

\title{A new family of MRD-codes}

\author{Bence Csajb\'ok, Giuseppe Marino, Olga Polverino, Corrado Zanella}

\date{ }

\newcommand{\cC}{{\mathcal C}}

\newcommand{\cN}{{\mathcal N}}

\newcommand{\cG}{{\mathcal G}}

\newcommand{\cF}{{\mathcal F}}
\newcommand{\cP}{{\mathcal P}}
\newcommand{\cH}{{\mathcal H}}

\newcommand{\F}{{\mathbb F}}

\newcommand{\la}{\langle}
\newcommand{\ra}{\rangle}

\newcommand{\V}{\mathbb V}

\newtheorem{theorem}{Theorem}[section]

\newtheorem{lemma}[theorem]{Lemma}
\newtheorem{corollary}[theorem]{Corollary}

\newtheorem{proposition}[theorem]{Proposition}

\newtheorem{conjecture}[theorem]{Conjecture}

\newtheorem{remark}[theorem]{Remark}

\DeclareMathOperator{\PG}{{PG}}

\DeclareMathOperator{\GL}{{GL}}

\begin{document}
\maketitle

\begin{abstract}
We introduce a family of linear sets of $\PG(1,q^{2n})$ arising from maximum scattered linear sets of pseudoregulus type of $\PG(3,q^{n})$.
For $n=3,4$ and for certain values of the parameters we show that these linear sets of $\PG(1,q^{2n})$ are maximum scattered and they yield
new MRD-codes with parameters $(6,6,q;5)$ for $q>2$ and with parameters $(8,8,q;7)$ for $q$ odd.
\end{abstract}



\section{Introduction}
\label{sec:Intro}

Linear sets are natural generalizations of subgeometries.
Let $\Lambda=\PG(V,\F_{q^n})\allowbreak=\PG(r-1,q^n)$, where $V$ is a vector space of dimension $r$ over $\F_{q^n}$. A point set $L$ of $\Lambda$ is said to be an \emph{$\F_q$-linear set} of $\Lambda$ of rank $k$ if it is
defined by the non-zero vectors of a $k$-dimensional $\F_q$-vector subspace $U$ of $V$, i.e.
\[L=L_U=\{\la {\bf u} \ra_{\mathbb{F}_{q^n}} \colon {\bf u}\in U\setminus \{{\bf 0} \}\}.\]
The maximum field of linearity of an $\F_q$-linear set $L_U$ is $\F_{q^t}$ if $t \mid n$ is the largest integer such that $L_U$ is an $\F_{q^t}$-linear set.

Two linear sets $L_U$ and $L_W$ of $\PG(r-1,q^n)$ are said to be \emph{$\mathrm{P\Gamma L}$-equivalent} (or simply \emph{equivalent}) if there is an element $\phi$ in $\mathrm{P\Gamma L}(r,q^n)$ such that $L_U^{\phi} = L_W$. It may happen that two $\F_q$--linear sets $L_U$ and $L_W$ of $\PG(r-1,q^n)$ are equivalent even if the two $\F_q$-vector subspaces $U$ and $W$ are not in the same orbit of $\Gamma \mathrm{L}(r,q^n)$ (see \cite{CSZ2015} and \cite{CSMP2016} for further details).

In \cite[Section 4]{Sh}, the author showed that scattered linear sets of $\PG(1,q^m)$ of rank $m$ yield $\F_q$-linear MRD-codes of dimension $2m$ and minimum distance $m-1$. Precisely, such codes are all $\F_q$-linear MRD-codes of dimension $2m$, minimum distance $m-1$ and middle nucleus of order $q^m$ (cf. Proposition \ref{prop:nuclei}). This result has been recently generalized in \cite{CSMPZ2016}. The number of non-equivalent MRD-codes obtained from a scattered linear set of $\PG(1,q^m)$ of rank $m$ was studied in
\cite[Section 5.4]{CSMP2016}. In \cite{Lu2017} the author investigated in detail the relationship between linear sets of $\PG(n-1,q^n)$ of rank $n$ and $\F_q$-linear MRD-codes.

So far, the known non-equivalent families of $\F_q$-linear MRD-codes of dimension $2m$, minimum distance $m-1$ and with \emph{middle nucleus} $\F_{q^m}$
(which is an invariant with respect to the equivalence on MRD-codes, see Section \ref{sec:MRD}) arise from the following maximum scattered $\F_q$--vector subspaces of $\F_{q^m}\times\F_{q^m}$:

\begin{enumerate}
\item $U_1:= \{(x,x^{q^s}) \colon x\in \F_{q^m}\}$, $1\leq s\leq m-1$ $\gcd(s,m)=1$ (\cite{BL2000}) gives Gabidulin codes when $s=1$, and generalized Gabidulin codes when $s>1$;
\item $U_2:= \{(x,\delta x^{q^s} + x^{q^{m-s}})\colon x\in \F_{q^m}\}$, $\N_{q^m/q}(\delta)\neq 1$ (\footnote{$\N_{q^m/q}(\cdot)$ denotes the norm function from $\F_{q^m}$ over $\F_q$.}), $\gcd(s,m)=1$ (\cite{LP2001} for $s=1$) gives MRD-codes found by Sheekey in \cite{Sh} as part of a larger family. The equivalence issue for these codes was studied also by Lunardon, Trombetti and Zhou in \cite{LTZ}.
\end{enumerate}

In this paper we present a family of $\F_{q}$-linear sets of rank $m$ of $\PG(1,q^{m})$, $m=2n$ and $n>1$, arising from $\F_q$-linear sets of $\PG(3,q^n)$ of pseudoregulus type. These linear sets are defined by the following $\F_q$-vector subspaces of $\F_{q^m}\times\F_{q^m}$:
\begin{equation}\label{form:newspace}
U_{b,s}:=\{(x,b x^{q^s} + x^{q^{s+n}}) \colon x\in \F_{q^{2n}}\}
\end{equation}
with $\N_{q^{2n}/q^n}(b)\neq 1$, $1\leq s\leq 2n-1$ and $\gcd(s,n)=1.$

We will show that each point of $L_{U_{b,s}}$ has weight at most 2 (cf. Proposition \ref{notgreaterthan2}) and when $L_{U_{b,s}}$ is scattered and $m>4$, then, as we will see in Section \ref{sec:MRD}, the corresponding MRD-code is not equivalent to any previously known MRD-code with the same parameters. Finally, in the last section, we exhibit for $m=6$ and $m=8$ infinite examples of scattered $\F_q$-subspaces of type $U_{b,s}$ and hence new infinite families of MRD-codes.

\section{Linear sets}
Let $L_U$ be an $\F_q$-linear set of $\Lambda=\PG(r-1,q^n)$, $q=p^h$, $p$ prime, of rank $k$. We point out that different vector subspaces can define the same linear set. For this reason a linear set and the vector space defining it must be considered as coming in pair.

Let  $\Omega=\PG(W,\F_{q^n})$ be a subspace of $\Lambda$, then $\Omega\cap L_U$ is an $\F_q$--linear set of $\Omega$ defined by the $\F_q$--vector
subspace $U\cap W$ and, if $w_{L_U}(\Omega):=\dim_{\F_q}(W\cap U)=i$, we say that $\Omega$ has {\it weight $i$} w.r.t. $L_U$. Hence a point of $\Lambda$ belongs to $L_U$ if and only if it has weight at least 1 and, if $L_U$ has rank $k$, then $|L_U|\leq  q^{k-1}+q^{k-2}+\dots+q+1$. For further details on linear sets see \cite{OP2010} and \cite{LV2013}.

An $\F_q$--linear set $L_U$ of $\Lambda$ of rank $k$ is {\em scattered} if all of its points have weight 1, or equivalently, if
$L_U$ has maximum size $q^{k-1}+q^{k-2}+\cdots+q+1$. The associated $\F_q$--vector subspace $U$ is said to be {\em scattered}.
A scattered $\F_q$--linear set of $\Lambda$ of highest possible
rank is a {\it maximum scattered $\F_q$--linear set} of $\Lambda$; see \cite{BL2000}. Maximum scattered linear sets have a lot of applications in Galois Geometry. For a recent survey on the theory of scattered spaces in Galois Geometry and its applications see \cite{Lavrauw}.

The rank of a scattered $\F_q$-linear set of $\PG(r-1,q^n)$, $rn$ even, is at most $rn/2$ (\cite[Theorems 2.1, 4.2 and 4.3]{BL2000}).
For $n=2$ scattered $\F_q$-linear sets of $\PG(r-1,q^2)$ of rank $r$ are the Baer subgeometries.
When $r$ is even there always exist scattered $\F_q$--linear sets of rank
$\frac{rn}2$ in $\PG(r-1,q^n)$, for any $n\geq 2$ (see \cite[Theorem 2.5.5]{LPhdThesis} for an explicit example).
Existence results were proved for $r$ odd, $n-1 \leq r$, $n$ even, and $q>2$ in \cite[Theorem 4.4]{BL2000}, but
no explicit constructions were known for $r$ odd, except for the case $r=3$, $n=4$, see \cite[Section 3]{BBL2000}.
Very recently in \cite[Theorem 1.2]{BGMP2015} and in \cite[Section 2]{CSMPZ2016} maximum scattered $\F_q$-linear sets of $\PG(r-1,q^n)$ of rank $rn/2$ have been constructed for any integers $r, n\geq 2$, $rn$ even, and for any prime power $q\geq 2$.

\subsection{\texorpdfstring{Scattered linear sets of pseudoregulus type in $\PG(3,q^n)$}{Scattered linear sets of pseudoregulus type in PG(3,qn)}}

In \cite{LuMaPoTr2014}, generalizing results contained in
\cite{MPT2007}, \cite{LMPT} and  \cite{LV}, a family of maximum
scattered linear sets of
$\PG(2h-1,q^n)$ of rank $hn$ ($h,n \geq 2$), called of {\it pseudoregulus type}, is
introduced. In particular, a maximum scattered $\F_q$--linear set $L_U$ of $\Lambda=\PG(3,q^n)$ of rank $2n$ is of {\it pseudoregulus type} if
(i) there exist $q^n+1$ pairwise disjoint lines of $L_U$ of weight $n$ w.r.t. $L_U$, say
$s_1,s_2,\dots,s_{q^n+1}$;\\ (ii) there exist exactly two skew lines
$t_1$ and $t_2$ of $\Lambda$, disjoint from $L_U$, such that
$t_j\cap s_i\neq \emptyset$ for each $i=1,\dots,{q^n+1}$ and for
each $j=1,2$.

The set of lines $\mathcal{P}_{L_U} = \{s_i \colon
i=1,\dots,q^n+1\}$ is called the {\it $\F_q$--pseudoregulus} (or
simply {\it pseudoregulus}) of $\Lambda$ associated with $L_U$ and
$t_1$ and $t_2$ are the  {\it transversal lines} of
$\mathcal{P}_{L_U}$ (or {\it transversal lines} of $L_U$). Note that
by \cite[Corollary 3.3]{LuMaPoTr2014}, if $n>2$ the pseudoregulus
$\cP_{L_U}$ associated with $L_U$ and its transversal lines are
uniquely determined.

In \cite[Sec. 2]{LMPT} and in \cite[Theorems 3.5 and 3.9]{LuMaPoTr2014}, $\F_q$--linear sets of pseudoregulus type of $\PG(2h-1,q^n)$ of rank $hn$ ($h,n \geq 2$) have been algebraically characterized. In particular, in $\PG(3,q^n)$ we have the following result.

\begin{theorem}[]\label{thm:algebraicpseudoregulus}
Let $t_1=\PG(U_1,\F_{q^n})$ and $t_2=\PG(U_2,\F_{q^n})$ be two
disjoint lines of $\Lambda=\PG(V,\F_{q^n})=\PG(3,q^n)$ and let
$\Phi_f$ be a strictly semilinear collineation between $t_1$ and
$t_2$ defined by the $\F_{q^n}$-semilinear map $f$ with companion automorphism an element $\sigma\in
Aut(\F_{q^n})$ such that $Fix(\sigma)=\F_q$. Then, for each $\rho
\in \F_{q^n}^*$, the set
\[L_{\rho,f} = \{\langle {\bf u}+\rho f({\bf u}) \rangle_{\F_{q^n}} \colon {\bf u}\in U_1\setminus\{{\bf 0}\}\}\]
is an $\F_q$-linear set of $\Lambda$ of pseudoregulus type whose
associated pseudoregulus is $\mathcal{P}_{L_{\rho,f}}=\{\langle P,
P^{\Phi_f}\rangle\,:\, P\in t_1\}$, with transversal lines $t_1$ and
$t_2$.

Conversely, each $\F_q$--linear set of pseudoregulus type of
$\Lambda=\PG(3,q^n)$ can be obtained as described above.
\end{theorem}

In \cite{LuMaPoTr2014},  $\F_q$-linear sets of pseudoregulus type
of the projective line $\Lambda=\PG(V,\F_{q^n})=\PG(1,q^n)$ ($n\geq
2$) are also introduced. Let $P_1=\langle {\bf w} \rangle$
and $P_2=\langle {\bf v} \rangle$ be two distinct points of
the line $\Lambda$ and let $\tau$ be an $\F_q$-automorphism of
$\F_{q^n}$ such that $Fix(\tau)=\F_q$; then for each $\rho \in
\F_{q^n}^*$ the set
\begin{equation}\label{pseudoline}
W_{\rho,\tau}=\{\lambda {\bf w} +
\rho\lambda^{\tau} {\bf v}  \colon \lambda \in
\F_{q^n}\},
\end{equation}
is an $\F_q$--vector subspace of $V$ of dimension $n$ and  $L_{\rho,\tau}:=L_{W_{\rho,\tau}}$ is a maximum scattered $\F_q$-linear set of $\Lambda$. The linear sets $L_{\rho,\tau}$  are called of {\it pseudoregulus type} and  the points $P_1$ and $P_2$ are their {\it transversal points}. Also, if $n>2$, then these transversal points are uniquely determined (\cite[Prop. 4.3]{LuMaPoTr2014}). For more details on such linear sets see \cite{CSZ2016}.
\textcolor{black}{ Also, by \cite[Remark 4.5]{LuMaPoTr2014}, if $L_U$ is an $\F_q$-linear set of pseudoregulus type of $\PG(3,q^n)$, and
$s$ is a line of weight $n$ w.r.t. $L_U$, then $L_U\cap s$ is an $\F_q$-linear set of pseudoregulus type of the line $s$ whose transversal points are the intersection points of $s$ with the transversal lines of
$\mathcal{P}_{L_U}$ (see also \cite[Prop. 2.5]{LaMaPoTr2013} and \cite[Theorem 2.8]{MaPo2015} for further details).}

\section{\texorpdfstring{Linear sets and dual linear sets in $\PG(1,q^n)$}{Linear sets and dual linear sets in PG(1,qn)}}\label{sec:dual}

Let  $\V=\mathbb F_{q^n} \times \mathbb F_{q^n}$ and let $L_U$ be an $\F_q$--linear set of rank $n$ of $\PG(1,q^n)=\PG(\V,\F_{q^n})$. We can always assume (up to a projectivity) that $L_U$ does not contain the point $\la(0,1)\ra_{\F_{q^n}}$. Then $U=U_f=\{(x,f(x)) \colon x\in \F_{q^n}\}$, for some $q$-polynomial $f(x)=\sum_{i=0}^{n-1} a_i x^{q^i}$ over $\F_{q^n}$. For the sake of simplicity we will write $L_f$ instead of $L_{U_f}$ to denote the linear set defined by $U_f$.

Consider the non-degenerate symmetric bilinear form of $\mathbb F_{q^n}$
over $\F_q$ defined by the following rule
\begin{equation}\label{form:angle}
<x,y>:=\Tr_{q^n/q}(xy). (\footnote{$\Tr_{q^n/q}(\cdot)$ denotes the trace function from $\F_{q^n}$ over $\F_q$.})
\end{equation}

\noindent Then the {\it adjoint map} $\hat f$ of an $\F_q$-linear map $f(x)=\sum_{i=0}^{n-1}a_i x^{q^i}$ of $\F_{q^n}$ (with respect to the bilinear form \eqref{form:angle}) is
\begin{equation}
\hat f(x):=\sum_{i=0}^{n-1}a_i^{q^{n-i}} x^{q^{n-i}}.
\end{equation}

Let $\eta: \V \times \V \longrightarrow \F_{q^{n}}$ be the  non-degenerate alternating bilinear  form of
$\V$ defined by $\eta((x,y), (u,v))=xv-yu$. Then $\eta$ induces a symplectic polarity $\tau$ on the line $\PG(\V,\F_{q^n})$ and
\begin{equation}\label{form:perp1}
\eta'((x,y), (u,v)):=\Tr_{q^n/q}(\eta((x,y), (u,v)))=\Tr_{q^n/q}(xv-yu)
\end{equation}
is a non-degenerate alternating bilinear form on $\V$, when $\V$ is regarded as a $2n$-dimensional vector space over $\F_{q}$.
We will always denote in the paper by $\perp$ and $\perp'$ the orthogonal complement maps defined by $\eta$ and $\eta'$ on the lattices of
the $\F_{q^n}$-subspaces and the $\F_{q}$-subspaces of $\V$, respectively. Direct calculation shows
that \begin{equation}\label{form:hatf} U_{f}^{\perp '}=U_{\hat f},\end{equation}
and the $\F_q$--linear set of rank $n$ of $\PG(\V,\F_{q^n})$ defined by the orthogonal complement
$U^{\perp'}$  is called  {\it the dual linear set of $L_U$} with respect to the polarity $\tau$.

Recall the following lemma.

\begin{lemma}[{\cite[Lemma 2.6]{BGMP2015}, \cite[Lemma 3.1]{CSMP2016}}]
\label{lem2} Let $L_f=\{\langle(x,f(x))\rangle_{\F_{q^n}} \colon x\in \F_{q^n}^*\}$ be an $\F_q$--linear set of $\PG(1,q^n)$ of rank $n$, with $f(x)$ a $q$-polynomial over $\F_{q^n}$, and let $\hat f$ be the adjoint of $f$ with respect to the bilinear form (\ref{form:angle}). Then for each point $P\in \PG(1,q^n)$ we have $w_{L_{f}}(P)=w_{L_{\hat f}}(P)$. In particular,
$L_{f}=L_{\hat f}$ and the maps defined by $f(x)/x$ and $\hat f(x)/x$  have
the same image.
\end{lemma}

\section{\texorpdfstring{From the geometry in $\PG(3,q^n)$ to the geometry in $\PG(1,q^{2n})$}{Geometry in PG(3,qn)}}

From now on, we will consider $\V=\F_{q^{2n}}\times \F_{q^{2n}}$ both as a 2-dimensional vector space over $\F_{q^{2n}}$ and as a 4-dimensional vector space over $\F_{q^n}$. In the former case the linear set of $\Sigma_1:=\PG(\V,\F_{q^{2n}})=\PG(1,q^{2n})$ defined by an $\F_q$-subspace $U \leq \V$ will be denoted as $L_U$, in the latter case the linear set of $\Sigma_3:=\PG(\V,\F_{q^n})=\PG(3,q^n)$ defined by $U$ will be denoted by $\bar L_U$.

Consider the following two skew lines of $\Sigma_3$:
$\ell_0:=\{\la(x,0)\ra_{\F_{q^{n}}}\colon x\in \F_{q^{2n}}^*\}$ and $\ell_1:=\{\la(0,y)\ra_{\F_{q^n}}\colon y\in \F_{q^{2n}}^*\}$.
By Theorem \ref{thm:algebraicpseudoregulus}, $\F_q$-linear sets of pseudoregulus type in $\Sigma_3$ with transversal lines $\ell_0$ and $\ell_1$ are of the form $\bar L_{f}:=\bar L_{U_f}$, where $U_f=\{(x,f(x))\colon x\in \F_{q^{2n}}\}$, and $f(x)$ is a strictly $\F_{q^n}$-semilinear invertible map of $\F_{q^{2n}}$ with companion automorphism $\sigma$, $Fix(\sigma)=\F_q$. It is easy to see that this happens if and only if $f(x)=\alpha x^{\sigma}+\beta x^{\sigma q^n}$, where $\sigma \colon x \mapsto x^{q^s}$, $1\leq s\leq 2n-1$, $\gcd(s,n)=1$, and $\N_{q^{2n}/q^n}(\alpha)\ne \N_{q^{2n}/q^n}(\beta)$. That is,
\begin{equation}\label{form:Uf}
U_f=\{(x,\alpha x^{\sigma}+\beta x^{\sigma q^n})\colon x\in \F_{q^{2n}}\},
\end{equation}
with the same conditions as above. In $\Sigma_1$ the $\F_q$-linear set $L_f:=L_{U_f}$ is not necessarily scattered, but as the next result shows, it cannot contain points with weight greater than two.

\begin{proposition}
\label{notgreaterthan2}
Each point of the $\F_q$-linear set $L_{f}$ of $\PG(1,q^{2n})$, $n\geq 2$, where
\[U_f=\{(x,f(x))\colon x\in \F_{q^{2n}}\},\]
with $f(x)=\alpha x^{\sigma}+\beta x^{\sigma q^n}$, $\sigma \colon x \mapsto x^{q^s}$, $1\leq s\leq 2n-1$,  $\gcd(s,n)=1$, and $\N_{q^{2n}/q^n}(\alpha)\ne \N_{q^{2n}/q^n}(\beta)$, has weight at most two.
\end{proposition}
\begin{proof}
We first recall that the pseudoregulus associated with $\bar L_f$ in $\Sigma_3=\PG(3,q^n)$ consists of $q^n+1$ lines, and these are the only lines with weight $n$ w.r.t. $\bar L_f$ (\cite[Prop. 3.2]{LuMaPoTr2014}).

Let $Q:=\la(x_0,f(x_0))\ra_{\F_{q^{2n}}}$ be a point of $L_{f}$.
In $\Sigma_3$ this point corresponds to a line $\ell_{Q}$ disjoint from both $\ell_0$ and $\ell_1$ and meeting at least one line of the pseudoregulus associated with $\bar L_f$, say $m$. Note that $w_{L_f}(Q)=w_{\bar L_f}(\ell_Q)$. By \cite[Theorem 5.1]{BlLa} a plane of $\Sigma_3$ has weight either $n$ or $n+1$ w.r.t. $\bar L_f$, hence if the weight of $Q$ w.r.t. $L_f$ is greater than one, then the plane $\pi$ of $\Sigma_3$ spanned by the lines $\ell_Q$ and $m$ has weight $n+1$. Since $\ell_Q \cap m$ is a point with weight one w.r.t. $\bar L_{f}$, the Grassmann formula gives that the weight of $\ell_Q$ w.r.t $\bar L_{f}$ is two and hence the weight of $Q$ w.r.t. $L_{f}$ is two.
\end{proof}

\section{\texorpdfstring{A family of $\F_q$-linear sets of $\PG(1,q^{2n})$}{A family of Fq-linear sets of PG(1,q2n)}}
\label{sec:family}

In this section we investigate the family of $\F_q$--linear sets of $\PG(1,q^{2n})$ defined by $\F_q$--vector subspaces of form (\ref{form:Uf}). Let $U_f$ and $U_g$ be two $\F_q$--vector subspaces of $\V=\F_{q^{2n}}\times \F_{q^{2n}}$ of form (\ref{form:Uf}), where $f(x)=\alpha x^{q^s}+\beta x^{q^{s+n}}$ and $g(x)=\alpha' x^{q^s}+\beta' x^{q^{s+n}}$ , with  $1\leq s\leq 2n-1$ and $\gcd(s,n)=1$.
Since we are interested in the study of scattered linear sets of $\PG(1,q^{2n})$ not of pseudoregulus type, we can assume $\alpha\beta\ne 0$ (cf. \cite[Sec. 4]{LuMaPoTr2014}). If $\N_{q^{2n}/q^n}(\alpha\beta')= \N_{q^{2n}/q^n}(\alpha'\beta)$ then there exists $a\in\F_{q^{2n}}^*$ such that $\beta\alpha'=\beta'\alpha a^{q^s(q^n-1)}$ and direct computations show that $U_f^\varphi=U_g$, where $$\varphi\colon (x,y)\in\V\mapsto (x a,y a^{q^s} \alpha'/\alpha)\in\V.$$ From the previous arguments it follows that $L_f$ is defined, up to the action of the group $\GL(2,q^n)$, by an $\F_q$--vector subspace of $\V$ of type
\begin{equation}\label{form:Ubs}
U_{b,s}:=\{(x,bx^{q^s}+x^{q^{s+n}})\colon x\in \F_{q^{2n}}\},
\end{equation}
with $b\in\F_{q^{2n}}^*$ and $1\leq s\leq 2n-1$ such that $\N_{q^{2n}/q^n}(b)\ne 1$ and $\gcd(s,n)=1$.
We will denote by $L_{b,s}$ the corresponding $\F_q$--linear set $L_{U_{b,s}}$.

Also we can restrict our study to the choice of the integers $s$' such that $1\leq s\leq n$ and $\gcd (s,n)=1$. Indeed, by using the notation of Section 3, we have
$$U_{b,s}^{\perp'}=\{(x,b^{q^{2n-s}}x^{q^{2n-s}}+x^{q^{n-s}})\colon x\in \F_{q^{2n}}\}=U_{b^{q^{2n-s}},2n-s}$$ and it can be easily seen that $U_{b,s}$ and $U_{b,s}^{\perp'}$ are equivalent via the linear invertible map $\phi\colon (x,y)\in\V\mapsto (\alpha y, \beta x)\in\V$, where $\alpha$ is any element satisfying $\alpha^{q^n-1}=-\frac 1 {b^{q^n-1}}$ and $\beta=(b^{2q^n}\alpha^{q^n}+\alpha)^{q^{n-s}}$.

\medskip

Moreover we have the following result.

\begin{proposition}
Two $\F_q$-subspaces $U_{b,s}$ and $U_{\bar b,\bar s}$ of $\V=\F_{q^{2n}}\times \F_{q^{2n}}$ of form \eqref{form:Ubs} with $b,\bar b\in\F_{q^{2n}}^*$,  $N_{q^{2n}/q^n}(b)\ne 1$, $N_{q^{2n}/q^n}(\bar b)\ne 1$,  $1\leq s, \bar s < n$ and $\gcd(n,s)=\gcd(n,\bar s)=1$, are ${\rm \Gamma L}(2,q^{2n})$-equivalent if and only if either $$s=\bar s \mbox{\quad and
\quad} \N_{q^{2n}/q^{n}}(\bar b)=\N_{q^{2n}/q^{n}}(b)^\sigma$$ or $$s+\bar s =n \mbox{\quad and\quad} \N_{q^{2n}/q^n}(\bar b)\N_{q^{2n}/q^n}(b)^\sigma=1,$$ for some automorphism $\sigma\in Aut(\F_{q^n})$.
\end{proposition}
\begin{proof}
$U_{b,s}$ and $U_{\bar b,\bar s}$ are ${\rm \Gamma L}(2,q^{2n})$-equivalent if and only if there exist elements $\alpha,\beta,\gamma,\delta \in\F_{q^{2n}}$, with $\alpha\delta\ne
\beta\gamma$ and an automorphism $\sigma\in Aut(\F_{q^{2n}})$ such that
$$
\forall \,x \in \F_{q^{2n}}, \exists \,y\in\F_{q^{2n}}:\quad \left(
  \begin{array}{cc}
    \alpha & \beta \\
    \gamma & \delta \\
  \end{array}
\right)\left(
         \begin{array}{c}
           x^\sigma \\
           (bx^{q^s}+x^{q^{s+n}})^{\sigma} \\
         \end{array}
       \right)=\left(
                 \begin{array}{c}
                   y \\
                   \bar by^{q^{\bar s}}+y^{q^{\bar s+n}} \\
                 \end{array}
               \right).$$
Put $z:=x^\sigma$, the last equation implies that for each
$z\in\F_{q^{2n}}$, there exists $y\in\F_{q^{2n}}$
such that
\begin{equation}\label{form-new-1}
\left\{\begin{array}{ll}
\alpha z+\beta (b^\sigma z^{q^s}+z^{q^{n+s}})= y,\\
\gamma z+\delta (b^\sigma z^{q^s}+z^{q^{n+s}})=\bar b y^{q^{\bar s}}+y^{q^{n+\bar s}}.
\end{array}\right.
\end{equation}
Putting the first in the second equation of System (\ref{form-new-1}), we get that
\begin{equation}\label{form-new-3}\gamma z+\delta (b^\sigma z^{q^s}+z^{q^{n+s}})=\bar b (\alpha z+\beta (b^\sigma z^{q^s}+z^{q^{n+s}}))^{q^{\bar s}}+(\alpha z+\beta (b^\sigma z^{q^s}+z^{q^{n+s}}))^{q^{n+\bar s}}\end{equation} for each $z\in\F_{q^{2n}}$.

If $s=\bar s$, since the monomials $z, z^{q^s}, z^{q^{2s}}, z^{q^{n+s}}, z^{q^{n+2s}}$ are pairwise distinct modulo $z^{q^{2n}}-z$, from the previous polynomial identity we get
\begin{equation}\label{form-new-2}
\left\{\begin{array}{ll}
\gamma=0\\
\delta b^\sigma=\bar b\alpha^{q^s}\\
\delta=\alpha^{q^{n+s}}\\
\bar b \beta^{q^s}b^{\sigma q^s}+\beta^{q^{n+s}}=0\\
\bar b\beta^{q^s}+\beta^{q^{n+s}}b^{\sigma q^{n+s}}=0.
\end{array}\right.
\end{equation}
Since $\N_{q^{2n}/q^{n}}(b)\ne 1$, System (\ref{form-new-2}) is equivalent to
\[
\left\{\begin{array}{ll}
\gamma=0\\
\beta=0\\
\delta b^\sigma=\bar b\alpha^{q^s}\\
\delta=\alpha^{q^{n+s}},\\
\end{array}\right.
\]
which admits solutions if and only if $\N_{q^{2n}/q^{n}}(\bar b)=\N_{q^{2n}/q^{n}}(b)^\sigma$, with $\sigma\in Aut(\F_{q^n})$.

If $s\neq \bar s$, since $1\leq s,\bar s< n$ and $\gcd (s,n)=\gcd(\bar s,n)=1$, we get
\[\{z^{q^s},z^{q^{\bar s}}\}\cap\{z,z^{q^{n+s}},z^{q^{n+ \bar s}},z^{q^{s+\bar s}},z^{q^{n+s+\bar s}}\}=\emptyset\]
modulo $z^{q^{2n}}-z$. Hence polynomial identity (\ref{form-new-3}) yields $\alpha=\delta=0$ and Equation (\ref{form-new-3}) becomes
\[
\gamma z=(\bar b \beta^{q^{\bar s}}b^{\sigma q^{\bar s}}+ \beta^{q^{n+\bar s}})z^{q^{s+\bar s}}+(\bar b \beta^{q^{\bar s}}+\beta^{q^{n+\bar s}} b^{\sigma q^{n+\bar s}})z^{q^{n+s+\bar s}}
\]
for each $z\in\F_{q^{2n}}$.
Also, since $s+\bar s<2n$, the monomials $z$ and $z^{q^{s+\bar s}}$ are different modulo $z^{q^{2n}}-z$. Hence, if $s+\bar s\ne n$ we immediately get $\gamma=0$, a contradiction. It follows that  $s+\bar s=n$ and comparing the coefficients of the terms of degree $1$ and $q^{s+\bar s}$ we get
\[
\left\{\begin{array}{ll}
\gamma=\bar b \beta^{q^{\bar s}}+\beta^{q^{n+\bar s}} b^{\sigma q^{n+\bar s}}\\
\bar b \beta^{q^{\bar s}}b^{\sigma q^{\bar s}}+ \beta^{q^{n+\bar s}}=0,\\
\end{array}\right.
\]
which admits solutions if and only if $\N_{q^{2n}/q^n}(\bar b b^{\sigma q^{\bar s}})=1$, i.e. if and only if
$\N_{q^{2n}/q^n}(\bar b)\N_{q^{2n}/q^n}(b^{q^{\bar s}})^\sigma=1$, for some automorphism $\sigma\in Aut(\F_{q^n})$.
\end{proof}
\bigskip

We finish this section by determining the linear automorphism group of $U_{b,s}$ and with some results on the geometric structure of a linear set $L_{b,s}$.

\begin{corollary}\label{prop:lingroupUbs}
The $\F_{q^{2n}}$-linear automorphism group $\cG_{b,s}$ of an $\F_q$--vector subspace $U_{b,s}$ of $\V=\F_{q^{2n}}\times \F_{q^{2n}}$ of form (\ref{form:Ubs}) consists of the following matrices
$$\left(
    \begin{array}{cc}
      \alpha & 0 \\
      0 & \alpha^{q^s} \\
    \end{array}
  \right),
$$ with $\alpha\in\F_{q^n}^*.$
\end{corollary}

\begin{proof}
In the previous theorem choosing $s=\bar s$ and $b= \bar b$, by System (\ref{form-new-2}) we get $\beta=\gamma=0$ and $\delta=\alpha^{q^s}=\alpha^{q^{n+s}}$. The assertion follows.
\end{proof}

The previous corollary allows us to prove the following result.

\begin{proposition}\label{prop:geomproper}
Let $L_{b,s}$ be the $\F_q$--linear set of $\PG(1,q^{2n})$ of rank $2n$ defined by an $\F_q$--vector subspace $U_{b,s}$ of type \eqref{form:Ubs} and let
$P\cG_{b,s}$ be the projectivity group induced on the line $\PG(1,q^{2n})$ by $\cG_{b,s}$. Then the following properties hold:
\begin{enumerate}[i)]
\item the linear collineation group $P\cG_{b,s}$ preserves $L_{b,s}$, it has order $\frac{q^n-1}{q-1}$, fixes the two points $\la (1,0)\ra_{\F_{q^{2n}}}$ and $\la (0,1)\ra_{\F_{q^{2n}}}$ and any other point--orbit has size $\frac{q^n-1}{q-1}$;
\item $L_{b,s}$ is a union of orbits of points under the $P\cG_{b,s}$--action;
\item all points of  $L_{b,s}$ belonging to the same $P\cG_{b,s}$--orbit have the same weight w.r.t. $L_{b,s}$.
\end{enumerate}
\end{proposition}
\begin{proof}
  Let $\phi_\lambda$ be the linear collineation of $P\cG_{b,s}$ induced by the element
$  \varphi_\lambda:=\left(
    \begin{array}{cc}
      \lambda & 0 \\
      0 & \lambda^{q^s} \\
    \end{array}
  \right) \in \cG_{b,s}
$, with $\lambda\in \F_{q^n}^*$.
   Since ${\rm Fix}(\sigma)\cap \F_{q^n}^*=\F_q$, the group $P\cG_{b,s}$ has order $\frac{q^n-1}{q-1}$. Also, it can be easily seen that if $P$ is a point of $\PG(1,q^{2n})$ different from $\la (1,0)\ra_{\F_{q^{2n}}}$ and $\la (0,1)\ra_{\F_{q^{2n}}}$, then $P^{\phi_\lambda}=P$ if and only if $\phi_\lambda$ is the identity map. Hence Statements $i)$ and $ii)$ follow.

 Let now $P=\la (x_0,f(x_0))\ra_{\F_{q^{2n}}}$ be a point of $L_{b,s}$, i.e. $f(x_0)=bx_0^{q^s}+x_0^{q^{n+s}}$.
 Then $P^{\phi_\lambda}=\la (\lambda x_0,f(\lambda x_0))\ra_{\F_{q^{2n}}}$ and
 \[w_{L_{b,s}}(P)=\dim_q(\la (x_0,f(x_0))\ra_{\F_{q^{2n}}}\cap U_{b,s})=\dim_q\varphi_\lambda(\la (x_0,f(x_0))\ra_{\F_{q^{2n}}}\cap U_{b,s})\]
\[=\dim_q\Big(\la (\lambda x_0,f(\lambda x_0))\ra_{\F_{q^{2n}}}\cap \varphi_\lambda(U_{b,s})\Big)\]
\[=\dim_q\Big(\la (\lambda x_0,f(\lambda x_0)\ra_{\F_{q^{2n}}}\cap U_{b,s}\Big)=w_{L_{b,s}}(P^{\phi_\lambda}),\]
 and Property $iii)$ is proved.
\end{proof}

From the previous proposition we get the following result.

\begin{corollary}
Let $L_{b,s}$ be the $\F_q$--linear set of $\PG(1,q^{2n})$ of rank $2n$ defined by an $\F_q$--vector subspace $U_{b,s}$ of type (\ref{form:Ubs}). The size of $L_{b,s}$ is a multiple of $\frac{q^n-1}{q-1}$. Furthermore, the set of points of weight 2 w.r.t. $L_{b,s}$  is a union of orbits under the action of the linear collineation group $P\cG_{b,s}$. \qed
\end{corollary}

\section{\texorpdfstring{Scattered $\F_q$-subspaces of type $U_{b,s}$ and the corresponding MRD-codes}{Scattered Fq-linear sets Lbs of PG(1,q2n) and the corresponding MRD--codes}}
\label{sec:MRD}

The set of $m \times n$ matrices $\F_q^{m\times n}$ over $\F_q$ is a rank metric $\F_q$-space
with rank metric distance defined by $d(A,B) = rk\,(A-B)$ for $A,B \in \F_q^{m\times n}$.
A subset $\cC \subseteq \F_q^{m\times n}$ is called a rank distance code (RD-code for short). The minimum distance of $\cC$ is
\[d(C) = \min_{{A,B \in \cC},\ {A\ne B}} \{ d(A,B) \}.\]

In \cite{Delsarte} the Singleton bound for an $m\times n$ rank metric code $\cC$ with minimum rank distance $d$ was proved:
\begin{equation}
\label{Singleton}
\#\cC \leq q^{\max \{m,n\}(\min \{m,n\}-d+1)}.
\end{equation}
If this bound is achieved, then $\cC$ is an MRD-code.
MRD-codes have various applications in communications and cryptography; see for instance \cite{Gabidulin, koetter_coding_2008}. More properties of MRD-codes can be found in \cite{Delsarte,Gabidulin,gadouleau_properties_2006,morrison_equivalence_2013}.

When $\cC$ is an $\F_q$-linear subspace of $\F_q^{m\times n}$, we say that $\cC$ is an $\F_q$-linear code and the
dimension $\dim_q (\cC)$ is defined to be the dimension of $\cC$ as a subspace over $\F_q$.
If $d$ is the minimum distance of $\cC$ we say that $\cC$ has parameters $(m,n,q;d)$.

The \emph{middle nucleus} of a code $\cC \subseteq \F_q^{m\times n}$ (cf. \cite{LTZ2}, or \cite{LN2016} where the term \emph{left idealiser} was used), is defined as
\[\cN(\cC):=\{Z\in \F_q^{m\times m} \colon ZC \in \cC \text{ for all } C\in \cC\},\]
and by \cite[Theorem 5.4]{LTZ2} it turns out to be a field of order at least $q$.

We will use the following equivalence definition for codes of $\F_q^{m \times m}$. If $\cC$ and $\cC'$ are two codes then they are equivalent if and only if
there exist two invertible matrices $A,B \in \F_q^{m \times m}$ and a field automorphism $\sigma$ such that
$\{A C^\sigma B \colon C\in \cC\}=\cC'$, or $\{A C^{T\sigma}B \colon C\in \cC\}=\cC'$, where $T$ denotes transposition.
The code $\cC^T$ is also called the \emph{adjoint} of $\cC$.

\medskip

In \cite[Section 5]{Sh} Sheekey showed that scattered $\F_q$-linear sets of $\PG(1,q^m)$ of rank $m$ yield $\F_q$-linear MRD-codes with parameters $(m,m,q;m-1)$.
We briefly recall here the construction from $\cite{Sh}$. Let $U_f$ be a maximum scattered $\F_q$--vector subspace, defined as above.
Then, after fixing an $\F_q$-bases for $\F_{q^m}$, the set of $\F_q$-linear maps of $\F_{q^m}$ 
\begin{equation}
\label{Cf}
\cC_f:=\{x\mapsto af(x)+bx \colon a,b \in \F_{q^m}\}
\end{equation}
corresponds to $m\times m$ matrices over $\F_q$ forming an $\F_q$-linear MRD-code with parameters $(m,m,q;m-1)$. Also, since $\cC_f$ is an $\F_{q^m}$-subspace of $End(\F_{q^m},\F_q)$, its middle nucleus $\cN(\cC_f)$ contains the set of scalar maps $\cF_m:=\{x\in\F_{q^m}\mapsto \alpha x\in\F_{q^m}\colon \alpha\in\F_{q^m}\}$, i.e. $|\cN(\cC_f)|\geq q^m$.

On the other hand  $\cN(\cC_f)$ is an $\F_q$-subspace of invertible maps together with the zero map (cf. \cite[Corollary 5.6]{LTZ2}), it is also an MRD-code with parameters $(m,m,q;m)$. Then \eqref{Singleton} gives $|\cN(\cC_f)|\leq q^m$, thus $\cN(\cC_f)=\cF_m$.

\medskip

Regarding the converse we can state the following.

\begin{proposition}\label{prop:nuclei}
If $\cC$ is an MRD-code with parameters $(m,m,q;m-1)$ and with middle nucleus isomorphic to $\F_{q^m}$, then $\cC$ is equivalent to some code $\cC_f$ (cf. \eqref{Cf}). 
\end{proposition}
\begin{proof}
By using a ring isomorphism between $\F_q^{m\times m}$ and $End(\F_{q^m},\F_q)$, we may suppose that $\cC\subset End(\F_{q^m},\F_q)$. 
Since $\cN(\cC)\setminus\{{\bf 0}\}$ and $\cF_m\setminus\{{\bf 0}\}$ are two Singer cyclic subgroups of $\GL(\F_{q^m},\F_q)$, there exists $H \in \GL(\F_{q^m},\F_q)$ such that
\[H^{-1}\circ \cN(\cC) \circ H = \cF_m,\]
see for example \cite[pg. 187]{Huppert}. With $\cC':=H^{-1}\circ \cC$ we can see that $\cN(\cC')=\cF_m$. 
It means that $\cC'$ is a 2-dimensional vector space over $\cF_m$ and hence it can be written as 
\[
\cC'=\{\alpha r(x)+\beta s(x)\colon \alpha,\beta\in\F_{q^m}\},
\] for some $q$-polynomials $r(x),s(x)$ over $\F_{q^m}$. Since each MRD-code with parameters $(m,m,q;m-1)$ contains invertible elements (cf. \cite[Lemma 2.1]{LTZ2}), we may take $h(x)\in \cC'$ invertible. Then $h^{-1} \circ \cC'$ has the desired form, i.e. $h^{-1} \circ \cC'=\cC_f$ for some $q$-polynomial $f(x)$ over $\F_{q^m}$. 
\end{proof}

\begin{proposition}\label{prop:knowncodes}
The known $\F_q$-linear MRD-codes with parameters\linebreak $(m,m,q;m-1)$ and with middle nucleus isomorphic to $\F_{q^m}$, up to equivalence, arise from one of the following maximum scattered subspaces of $\F_{q^{m}}\times\F_{q^{m}}$:
\begin{enumerate}
\item $U_1= \{(x,x^{q^s}) \colon x\in \F_{q^m}\}$, $1\leq s\leq m-1$ $\gcd(s,m)=1$.
\item $U_2= \{(x,\delta x^{q^s} + x^{q^{m-s}})\colon x\in \F_{q^m}\}$, $\N_{q^m/q}(\delta)\neq 1$, $\gcd(s,m)=1$.
\end{enumerate}
\end{proposition}
\begin{proof}
The known $\F_q$-linear MRD-codes with parameters $(m,m,q;m-1)$, written as $\F_q$-linear maps over $\F_{q^m}$, are of the form
\[\cH_{2,s}(\mu,h):=\{x \mapsto a_0x+a_1x^{q^s}+\mu a_0^{q^h} x^{q^{2s}} \colon a_0,a_1\in \F_{q^m}\},\]
with $\gcd(s,m)=1$ and $\N_{q^{sm}/q^s}(\mu)\neq 1$.

By \cite[Corollary 5.9]{LTZ2} the middle nuclei of the codes $\cH_{2,s}(\mu,h)$ are isomorphic to $\F_{q^m}$ if and only if
$\mu=0$ or $m \mid 2s-h$. In the former case we obtain generalized Gabidulin codes arising from maximum scattered linear sets of pseudoregulus type, i.e. from maximum scattered subspaces of $\F_{q^{m}}\times\F_{q^{m}}$ of type $U_1$. If $m \mid 2s-h$, by \cite[Proposition 4.3]{LTZ} the adjoint code of $\cH_{2,s}(\mu,h)$ is equivalent to $\cH_{2,s}(1/\mu,2s-h)=\cH_{2,s}(1/\mu,0)$ and direct computations show that such a code is equivalent to a code arising from a maximum scattered subspace of type $U_2$. The assertion follows from the fact that the families of MRD-codes arising from maximum scattered subspaces of type $U_1$ and $U_2$, respectively, are both closed under the adjoint operation (following the terminology of \cite{Sh, Kantor, LuMaPoTr2011}, the adjoint code of $\cC_f$ is $\cC_{\hat{f}}$).

\end{proof}

Put $m=2n$, $n>1$ in the previous proposition. Note that if $n=2$ then a scattered $\F_q$--vector subspace $U_{b,s}$ (which means $\N_{q^{4}/q}(b)\ne 1$, cf. \cite{CSZ2017}) is of type either $U_2$ or $U_2^{\perp'}$. Now, we are able to prove that MRD-codes arising from scattered subspaces of form (\ref{form:Ubs}) with $n>2$ are new.

By using the same arguments as in Corollary \ref{prop:lingroupUbs}, the linear automorphism group $\cG_{i}$ of $U_i$, $i\in\{1,2\}$, is
\[\cG_1=\Bigg\{\left(
                  \begin{array}{cc}
                    a & 0 \\
                    0 & a^{q^s} \\
                  \end{array}
                \right)\ \colon a\in\F_{q^{2n}}^*\Bigg\},
      \quad\quad\cG_2=\Bigg\{\left(
                  \begin{array}{cc}
                    a & 0 \\
                    0 & a^{q^s} \\
                  \end{array}
                \right)\ \colon a\in\F_{q^{2}}^*\Bigg\}.
\]

This allows us to prove the following:
\begin{theorem}
\label{thm:new}
If $n>2$, the  $\F_q$--vector subspace of $\F_{q^{2n}}\times\F_{q^{2n}}$ \[U_{b,s}=\{(x,bx^{q^s}+x^{q^{s+n}})\colon x\in \F_{q^{2n}}\},\] with $b\in\F_{q^{2n}}^*$ and $1\leq s\leq n-1$ such that $\N_{q^{2n}/q^n}(b)\ne 1$ and $\gcd(s,n)=1$, is not equivalent to any subspace $U_i$, $i\in\{1,2\}$, under the action of the group $\Gamma {\rm L}(2,q^{2n})$.
\end{theorem}
\begin{proof}
If there exists an element $\varphi\in\Gamma{\rm L}(2,q^{2n})$ such that $U_{b,s}^\varphi=U_i$, for some $i\in\{1,2\}$, then the corresponding linear automorphism groups will be isomorphic via the map
$$\omega\in \cG_{b,s}\mapsto \varphi\circ\omega\circ\varphi^{-1}\in\cG_i,$$ but this is a contradiction by comparing the sizes of the related groups (cf. Corollary \ref{prop:lingroupUbs}).
\end{proof}

Let $\cC_f$ and $\cC_g$ be two MRD-codes arising from maximum scattered subspaces $U_f$ and $U_g$ of $\F_{q^m}\times \F_{q^m}$.
In \cite[Theorem 8]{Sh} the author showed that there exist invertible matrices $A$, $B$ such that $A \cC_f B=\cC_g$ if and only if $U_f$ and $U_g$ are  $\Gamma\mathrm{L}(2,q^m)$-equivalent. Hence, by Theorem \ref{thm:new}, we get the following result.

\begin{theorem}\label{thm:newMRD}
If $n>2$, the linear MRD-code of dimension $4n$ and minimum distance $2n-1$ arising from a scattered $\F_q$--vector subspace $U_{b,s}=\{(x,bx^{q^s}+x^{q^{s+n}})\colon x\in \F_{q^{2n}}\}$ of $\F_{q^{2n}}\times\F_{q^{2n}}$ is not equivalent to any previously known MRD-code with the same parameters.\qed
\end{theorem}

In the next section we will show that when $n=3$ and $q>2$ and when $n=4$ and $q$ is odd there exist values of $b$ and $s$ for which the $\F_q$-subspace $U_{b,s}$ of $\F_{q^{2n}} \times \F_{q^{2n}}$ is scattered, and from the above arguments the corresponding MRD-codes are new.

\section{New maximum scattered subspaces}

\subsection{\texorpdfstring{The $n=3$ case}{The n=3 case}}

We want to show that there exists $b\in\F_{q^6}^*$ such that
\[U_{b,1}:=\{(x,b x^q + x^{q^4})\colon x\in \F_{q^6}\}\]
is a maximum scattered $\F_q$-subspace.

$U_{b,1}$ is scattered if and only if for each $m\in \F_{q^6}$
\[\frac{b x^q + x^{q^4}}{x}=-m\]
has at most $q$ solutions. Those $m$ which admit exactly $q$ solutions correspond to points $\la(1,-m)\ra_{\F_{q^6}}$ of $L_{U_{b,1}}$ with weight one.
It follows that $U_{b,1}$ is scattered if and only if for each $m\in \F_{q^6}$ the kernel of
\[r_{m,b}(x):=mx+b x^q+x^{q^4}\]
has dimension less than two, or, equivalently, the Dickson matrix
\[D_{m,b}:=
\begin{pmatrix}
m&b&0&0&1&0\\
0&m^q&b^q&0&0&1\\
1&0&m^{q^2}&b^{q^2}&0&0\\
0&1&0&m^{q^3}&b^{q^3}&0\\
0&0&1&0&m^{q^4}&b^{q^4}\\
b^{q^5}&0&0&1&0&m^{q^5}\\
\end{pmatrix}
\]
associated to $r_{m,b}(x)$ has rank at least five (cf. \cite[Proposition 4.4]{WL}). Equivalently, $D_{m,b}$ has a non-zero $5 \times 5$ minor.
We will denote by $M_{i,j}$ the determinant of the matrix obtained from $D_{m,b}$ by removing the $i$-th row and the $j$-th column.
We will use the following:
\begin{equation}
\label{minor1}
M_{6,1}=b^{q^2}-b^{1+q^2+q^3}-b^{q+q^2+q^4}+b^{1+q+q^2+q^3+q^4}-b^{q^4}m^{q+q^2+q^3}-bm^{q^2+q^3+q^4},
\end{equation}
\begin{equation}
\label{minor5}
M_{6,5}=-b^{q^2}m+b^{q+q^2+q^4}m-bm^{q^3}+b^{1+q+q^4}m^{q^3}+b^{q^4}m^{1+q+q^2+q^3}.
\end{equation}

We will show that for certain choices of $b$ and $q$ there is no $m\in \F_{q^6}$ such that both of the above expressions are zero.

\begin{theorem}
\label{THM1}
For $q>4$ we can always find $b\in \F_{q^2}^*$, such that $U_{b,1}$ is a maximum scattered $\F_q$-subspace of $\F_{q^6}\times\F_{q^6}$.
\end{theorem}
\begin{proof}
We want to find $b\in \F_{q^2}^*$ such that at least one of \eqref{minor1} and \eqref{minor5} is non-zero.
Suppose the contrary, i.e. for each $b\in \F_{q^2}$:
\begin{equation}
\label{eq1}
0=b(1-2b^{q+1}+b^{2q+2}-m^{q+q^2+q^3}-m^{q^2+q^3+q^4}),
\end{equation}
\begin{equation}
\label{eq3}
0=b(-m+b^{q+1}m-m^{q^3}+b^{q+1}m^{q^3}+m^{1+q+q^2+q^3}).
\end{equation}
Put $x=m^{1+q+q^2}$ and $z=1-b^{q+1}$. Obviously $z\ne 1$ and dividing \eqref{eq1} by $b$ gives
\begin{equation}
\label{f1}
z^2=x^q+x^{q^2},
\end{equation}
multiplying \eqref{eq3} by $m^{q^4+q^5}/b$ gives
\begin{equation}
\label{f3}
z(x^{q^3}+x^{q^4})=x^{q^3+1}.
\end{equation}

Since $b\in \F_{q^2}$, it follows that $b^{q+1}\in \F_q$ and hence $z\in \F_q$. Then \eqref{f1} yields $x^q+x^{q^2}\in \F_{q}$ and hence
$x\in \F_{q^2}$. Then \eqref{f1} and \eqref{f3} give:
\begin{equation}
\label{g1}
z^2=x+x^q,
\end{equation}
\begin{equation}
\label{g2}
z^3=x^{q+1}.
\end{equation}
Thus $x$ and $x^q$ are roots of the equation
\begin{equation}
\label{quadr}
X^2-z^2X+z^3=0.
\end{equation}
From now on we distinguish two cases according to the parity of $q$. First suppose $q$ odd.
If \eqref{quadr} can be solved in $\F_q$, then $x=x^q\in \F_q$ and hence \eqref{g1} and \eqref{g2} give $z=x=0$, or $z=4$, $x=8$.
If we can find $z\in \F_q \setminus \{0,1,4\}$ such that \eqref{quadr} has roots in $\F_q$, then we obtain a contradiction meaning that the two minors in consideration cannot vanish at the same time. Then $U_{b,1}$ is scattered for each $b\in \F_{q^2}$ which satisfies $1-b^{q+1}=z$. Equation \eqref{quadr} has roots in $\F_q$ if and only if $z^4-4z^3$ is a square, hence, when $z^2-4z$ is a square. Note that $z=2$ gives $z^2-4z=-4$, which is always a square when $q \equiv 1 \pmod 4$. So from now on, we may assume $q \equiv 3 \pmod 4$ and hence $q\geq 7$.
Consider the conic $\cC$ of $\PG(2,q)$ with equation $X_0^2-4X_0X_2-X_1^2=0$. It is easy to see that $\cC$ is always non-singular, and that the line with equation $X_0=0$ is a tangent to $\cC$. For $q\geq 7$ $\cC$ has more than 7 points and hence we can find a point of $\cC$ not on the lines $X_0=0$, $X_0-4X_2=0$, $X_0-X_2=0$ and $X_2=0$. It means that we can always find a point $\la(x_0,x_1,1)\ra_{\F_q} \in \PG(2,q)$ such that $x_0^2-4x_0=x_1^2$ and $x_0\in \F_q \setminus \{0,1,4\}$. It follows that we can always find $z$, and hence $b$, with the given conditions.

Now consider the case when $q$ is even. For $z\neq 0$ \eqref{quadr} has a solution in $\F_q$ if and only if the $S$-invariant of the equation, that is $\Tr_{q/2}(1/z)$, equals to zero. If there is a solution in $\F_q$, then \eqref{g1} and \eqref{g2} give $z=0$, so it is enough to prove that there exists $z\in\F_q\setminus \{0,1\}$, such that $\Tr_{q/2}(1/z)=0$. The existence of such $z$ gives a contradiction meaning that the two minors in consideration cannot vanish at the same time.
The equation $\Tr_{q/2}(x)=0$ has $q/2$ pairwise distinct roots in $\F_q$,
thus $\Tr_{q/2}(1/z)=0$ has $q/2-1$ non-zero solutions. It follows that for $q \geq 8$ we can find such $z$.
\end{proof}

\subsection{\texorpdfstring{The $n=4$ case}{The $n=4$ case}}

We will show that there exists $b\in\F_{q^8}^*$ such that
\[U_{b,1}:=\{(x,b x^q + x^{q^5})\colon x\in \F_{q^8}\}\]
is a maximum scattered $\F_q$-subspace for each odd $q$.

$U_{b,1}$ is scattered if and only if for each $m\in \F_{q^8}$
\[\frac{b x^q + x^{q^5}}{x}=-m\]
has at most $q$ solutions. Those $m$ which admit exactly $q$ solutions correspond to points $\la(1,-m)\ra_{\F_{q^8}}$ of $L_{U_{b,1}}$ with weight one.
It follows that $U_{b,1}$ is scattered if and only if for each $m\in \F_{q^8}$ the kernel of
\[r_{m,b}(x):=mx+b x^q+x^{q^5}\]
has dimension less than two, or, equivalently, the Dickson matrix
\[D_{m,b}:=
\begin{pmatrix}
m&b&0&0&0&1&0&0\\
0&m^q&b^q&0&0&0&1&0\\
0&0&m^{q^2}&b^{q^2}&0&0&0&1\\
1&0&0&m^{q^3}&b^{q^3}&0&0&0\\
0&1&0&0&m^{q^4}&b^{q^4}&0&0\\
0&0&1&0&0&m^{q^5}&b^{q^5}&0\\
0&0&0&1&0&0&m^{q^6}&b^{q^6}\\
b^{q^7}&0&0&0&1&0&0&m^{q^7}\\
\end{pmatrix}
\]
of $r_{m,b}(x)$ has a non-zero $7 \times 7$ minor. If we remove the first two columns and last two rows of the above matrix, then the remaining $6 \times 6$ submatrix $M$ has determinant $(b^{q+q^5}-1)m^{q^3+q^4}$. It follows that with $\N_{q^8/q^4}(b)\neq 1$ the only point of $L_{U_{b,s}}$ with weight larger than 2 is $\la(1,0)\ra_{\F_{q^8}}$. On the other hand, it is easy to see that $\la(1,0)\ra_{\F_{q^8}}$ is a point of $L_{U_{b,s}}$ if and only if $\N_{q^8/q^4}(b)=1$.

We will denote by $M_{i,j}$ the determinant of the matrix obtained from $D_{m,b}$ by cancelling the $i$-row and the $j$-th column.
We will use the following:
\begin{equation}
\label{minor2}
M_{8,2}=(b^{1+q^4}-1)^{q+q^2}(b^{q^3+q^4}m+m^{q^4})+m^{1+q^3+q^4+q^5}(b^{q^6}m^{q^2}+b^qm^{q^6}).
\end{equation}

\begin{theorem}
\label{THM2}
For odd $q$ and $b^2=-1$ the $\F_q$-subspace $U_{b,1}$ is maximum scattered in $\F_{q^8}\times\F_{q^8}$.
\end{theorem}
\begin{proof}
We will show that there is no $m\in \F_{q^8}^*$ such that \eqref{minor2} vanishes.
Applying $b^2=-1$, the vanishing of \eqref{minor2} would give
\begin{equation}
\label{eq2}
0=4(b^{q+1}m+m^{q^4})+m^{1+q^3+q^4+q^5}(b m^{q^2}+b^qm^{q^6}).
\end{equation}
Now we distinguish two cases, according to $b\in \F_q$ (i.e., $q \equiv 1 \pmod 4$), or $b\in \F_{q^2} \setminus \F_q$ (i.e., $q \equiv 3 \pmod 4$). First suppose that the former case holds. Then
\begin{equation}
\label{eqe2}
0=4(-m+m^{q^4})+bm^{1+q^3+q^4+q^5}(m^{q^2}+m^{q^6}).
\end{equation}
Considering the $\F_{q^8}\rightarrow \F_{q^4}$ trace of both sides of \eqref{eqe2} and using the $\F_{q^4}$-linearity of this function, it follows that $\Tr_{q^8/q^4}(m^{q^3+q^5})=0$.
It is easy to see that $\Tr_{q^8/q^4}(x)=\Tr_{q^8/q^4}(y)=0$ implies $xy\in \F_{q^4}$ for any two $x,y\in \F_{q^8}$, thus
$m^{q^3+q^5}m^{q^2+q^4}$ and $m^{q^3+q^5}m^{q^4+q^6}$ are in $\F_{q^4}$. It follows that
$bm^{1+q^3+q^4+q^5}(m^{q^2}+m^{q^6})=m \lambda$ for some $\lambda\in \F_{q^4}$ and hence \eqref{eqe2} gives
$m^{q^4-1} \in \F_{q^4}$. But also $m^{q^4+1}\in \F_{q^4}$ and hence $m^2\in \F_{q^4}$ giving either $m\in \F_{q^4}$, or $\Tr_{q^8/q^4}(m)=0$, but \eqref{eqe2} gives $m=0$ in both cases.

Now consider the $b\in \F_{q^2} \setminus \F_q$ case. Then $b^{q+1}=1$ and $b^q=-b$, thus \eqref{eq2} gives
\begin{equation}
\label{eqf2}
0=4(m+m^{q^4})+bm^{1+q^3+q^4+q^5}(m^{q^2}-m^{q^6}).
\end{equation}
Since $4(m+m^{q^4})\in \F_{q^4}$ and $bm^{1+q^4}\in \F_{q^4}$, it follows that
$m^{q^3+q^5}(m^{q^2}-m^{q^6})\in \F_{q^4}$.
It is easy to see that $\Tr_{q^8/q^4}(x)=0$ and $xy\in \F_{q^4}$ implies $\Tr_{q^8/q^4}(y)=0$ for any two $x,y\in \F_{q^8}$, thus
$\Tr_{q^8/q^4}(m^{q^3+q^5})=0$. Then, as in the previous case, $m^2\in \F_{q^4}$ follows, which gives a contradiction.
\end{proof}


\begin{remark}
It follows from Theorem \ref{thm:new} that the maximum scattered subspaces of this section are new, i.e. they cannot be obtained from previously known maximum scattered subspaces under the action of $\Gamma \mathrm{L}(2,q^n)$, $n=6,8$.

The question of the equivalence of the corresponding linear sets under the action of the group $\mathrm{P}\Gamma \mathrm{L}(2,q^n)$ is addressed in \cite{CSMaZu}.
\end{remark}

\begin{remark}
Computations with {\tt GAP} yield the following results.

With respect to the cases not covered by Theorem \ref{THM1}: there exist $b\in \F_{q^6}^*$ such that the subspace
$\{(x,bx^q+x^{q^4}) \colon x\in \F_{q^6}\}$ is scattered in $\F_{q^6}\times \F_{q^6}$ also for $q\in \{3,4\}$, but not for $q=2$.

With respect to Theorem \ref{THM2}: for $q\leq 8$, $q$ even, there is no $b\in \F_{q^8}^*$ such that
$\{(x,bx^q+x^{q^5}) \colon x\in \F_{q^8}\}$ is scattered in $\F_{q^8}\times \F_{q^8}$ and for $q\leq 11$, $q$ odd, the corresponding subspace is scattered if and only if $b^{q^4+1}=-1$. According to the first paragraph of Section \ref{sec:family}, each of these subspaces is equivalent to the scattered subspace found in Theorem \ref{THM2}.

There is no $b\in \F_{q^{2n}}^*$ such that
$\{(x,bx^{q^s}+x^{q^{n+s}}) \colon x\in \F_{q^{2n}}\}$, $\gcd(s,n)=1$, is scattered in $\F_{q^{2n}}\times \F_{q^{2n}}$ when
$q\leq 5$ and $n\in \{5,6,7,8\}$, or $q=7$ and $n\in \{5,6,7\}$, or $q=7$ and $n=8$, or $q=8$ and $n=5$.
\end{remark}

\begin{conjecture}
According to the first paragraph of Section \ref{sec:family}, $f_1(x)=b_1 x^{q}+x^{q^4} \in \F_{q^6}[x]$ and $f_2(x)=b_2 x^{q}+x^{q^4} \in \F_{q^6}[x]$
define equivalent subspaces when $\N_{q^6/q^3}(b_1)=\N_{q^6/q^3}(b_2)$. We conjecture that the size of the set
\[ \{ \N_{q^6/q^3}(b) \colon f(x)=bx^{q}+x^{q^4} \text{ defines a maximum scattered $\F_q$-space $U_{b,1}$} \}\]
is $\lfloor(q^2+q+1)(q-2)/2\rfloor$, and hence there might be further examples of maximum scattered subspaces in this family.
By {\tt GAP} we verified this conjecture for $q\leq 32$.
\end{conjecture}

\begin{remark}
The maximum number of directions determined by an $\F_q$-linear function over $\F_{q^n}$ is $(q^n-1)/(q-1)$. Also, the maximum size of an $\F_q$-linear blocking set of R\'edei type of $\PG(2,q^n)$ is $q^n+(q^n-1)/(q-1)$. According to \cite[Section 5.3]{CSMP2016} our new examples of maximum scattered spaces yield new examples of functions and of blocking sets which attain these bounds.

In \cite[pg. 132]{GM2014} the maximal cardinality of the image set $\mathrm{Im}(L(x)/x)$ is considered (with $x\mapsto 1/x$ defined to take 0 to 0), where $L(x)$ is an $\F_p$-linear function over $\F_q$, $p$ is a prime and $q$ is a power of $p$. If for some invertible $p$-polynomial $f$, the subspace  $U_f=\{(x,f(x))\colon x\in \F_q\}$ is scattered, then the cardinality of $\mathrm{Im}(L(x)/x)$ reaches its maximum, which is $1+(q-1)/(p-1)$. It follows that the maximum scattered subspaces constructed in this paper yield such functions.
\end{remark}

 \subsection*{Acknowledgement}
The research  was supported by Ministry for Education, University and Research of Italy MIUR (Project
PRIN 2012 "Geometrie di Galois e strutture di incidenza") and by the Italian National
Group for Algebraic and Geometric Structures and their Applications (GNSAGA
- INdAM).

\bigskip

\noindent Bence Csajb\'ok\\
MTA--ELTE Geometric and Algebraic Combinatorics Research Group\\
ELTE E\"otv\"os Lor\'and University, Budapest, Hungary\\
Department of Geometry\\
1117 Budapest, P\'azm\'any P.\ stny.\ 1/C, Hungary\\
{{\em csajbok.bence@gmail.com}}

\bigskip

\noindent Giuseppe Marino, Olga Polverino\\
Dipartimento di Matematica e Fisica,\\
Universit\`a degli Studi della Campania ``Luigi Vanvitelli'',\\
Viale Lincoln 5, I-\,81100 Caserta, Italy\\
{{\em giuseppe.marino@unicampania.it}, {\em olga.polverino@unicampania.it}}
\bigskip

\noindent Corrado Zanella\\
Dipartimento di Tecnica e Gestione dei Sistemi Industriali,\\
Universit\`a di Padova,\\
Stradella S. Nicola, 3, I-36100 Vicenza, Italy\\
{{\em corrado.zanella@unipd.it}}

\bigskip
\end{document}